\documentclass[reqno]{amsart}
\RequirePackage[colorlinks,citecolor=black,urlcolor=blue]{hyperref}
\usepackage{enumitem}
\usepackage[usenames,dvipsnames]{xcolor}
\usepackage{hyperref}
\hypersetup{%
	colorlinks=true, linkcolor=blue,
	citecolor=ForestGreen
}
\usepackage[paper=letterpaper,margin=.9in]{geometry} 
\usepackage[utf8]{inputenc}
\usepackage{lmodern}
\usepackage{array}
\usepackage{adjustbox}
\usepackage{graphicx}
\usepackage{amssymb, amsfonts, amsthm, amsmath}
\usepackage{float}
\usepackage{comment}
\usepackage{tikz}
\usepackage{bbm}
\usepackage{xcolor} 
\usetikzlibrary{decorations.markings}
\usetikzlibrary{calc,angles,positioning}
\usetikzlibrary{decorations.pathreplacing}
\usepackage{chngcntr}
\usepackage{apptools}
\theoremstyle{plain}
\newtheorem{theorem}{Theorem}[section]

\newtheorem{defin}[theorem]{Definition}
\newtheorem{notation}[theorem]{Notation}
\newtheorem{prop}[theorem]{Proposition}

\newtheorem{lemma}[theorem]{Lemma}

\newtheorem{remark}[theorem]{Remark}

\newcommand{\lb}{\bar{\lambda}}

\numberwithin{equation}{section}
\begin{document}
\title[Exact Multi-Point Correlations in the SHE
for Strictly Sublinear Coordinates]{Exact Multi-Point Correlations in the Stochastic Heat Equation
for Strictly Sublinear Coordinates}
\author{Pierre Yves Gaudreau Lamarre}
\address[Pierre Yves Gaudreau Lamarre]{Mathematics Department, Syracuse University, Syracuse, NY 13244}
\email{pgaudrea@syr.edu}
\author{Yier Lin}
\address[Yier Lin]{Department of Statistics, University of Chicago, Chicago, IL 60637}
\email{ylin10@uchicago.edu}

\begin{abstract}
We consider the Stochastic Heat Equation (SHE)
in $(1+1)$ dimensions with delta Dirac initial data and spacetime white noise.
We prove exact large-time asymptotics for multi-point correlations of the SHE for
strictly sublinear space coordinates. The sublinear condition is
optimal, in the sense that different asymptotics are known to occur when the space coordinates
grow linearly \cite[Theorem 1.1]{Lin}.
Lastly, a notable feature of our result is that it confirms the connection between multi-point correlations in the SHE
and the ground state of the Hamiltonian of the delta-Bose gas.
\end{abstract}

\maketitle

\section{Introduction}

\subsection{Background}

In this note, we study the Stochastic Heat Equation (SHE)
\begin{equation}
\label{Equation: SHE}
\partial_t  Z = \tfrac{1}{2} \partial_{xx} Z + \xi Z,
\qquad Z=Z(t,x),~(t,x)\in\mathbb R_+\times\mathbb R
\end{equation}
starting from Dirac delta initial data $Z(0, \cdot) = \delta(\cdot)$,
where $\xi=\xi(t,x)$ is a spacetime white noise.
Informally, $\xi$ is defined as the centered Gaussian process with covariance
\[\mathbb E[\xi(s,x)\xi(t,y)]=\delta(t-s)\delta(y-x).\]
See, e.g., \cite[Sections 2.1--2.6]{Quastel} for a survey of the solution theory of this object.
The SHE is an object of fundamental importance in stochastic analysis and mathematical physics
due to its connections with various physical models, such as random polymers
and the KPZ equation; we refer to
\cite{Cor12, Quastel} for a detailed exposition of these (and more) connections. In this note, we are interested in the occurrence of intermittency in
$Z(t,\cdot)$ for large $t$. Informally, intermittency refers to the observation that, as $t\to\infty$,
the SHE's solution tends to concentrate most of its mass in tall and and narrow peaks separated by deep valleys.

The rigorous study of intermittency in the SHE began with the pair of
articles \cite{BertiniCancrini,BertiniGiacomin}, both of which followed the methodology
to prove intermittency outlined in \cite[Page 229]{Molchanov}: On the one hand, in \cite[(2.40)]{BertiniCancrini} the authors
obtained an explicit expression for the $n$-moment Lyapunov exponents
\begin{equation*}
\mathfrak L_n:=\lim_{t\to\infty}\frac{\log\mathbb E[Z(t,x)^n]}{t}
\end{equation*}
for all $n\in \mathbb{Z}_{\geq 1}$ under the assumption of a constant initial condition $Z(0,\cdot)=c>0$;
namely
\begin{equation}
\label{Equation: SHE Moment Lyapunov}
\mathfrak L_n=\frac{ n(n^2-1)}{24}.
\end{equation}
Thanks to a simple ergodic theorem/Markov's inequality argument outlined in \cite[(2.28)--(2.38)]{BertiniCancrini},
the fact that the moment Lyapunov exponents in \eqref{Equation: SHE Moment Lyapunov} satisfy $\mathfrak L_1<\mathfrak L_2/2<\mathfrak L_3/3<\cdots$
implies that for any $\alpha>0$, there exist small islands (which occupy an exponentially small proportion
of space that can be quantified using the exponents $\mathfrak L_n$) on which $Z(t,\cdot)$ exceeds $e^{\alpha t}$.

On the other hand, in \cite[Theorem 1.1]{BertiniGiacomin}, the authors show that 
for every smooth and compactly supported $\phi:\mathbb R\to\mathbb R$, one has
\begin{equation}
\label{Equation: SHE Sample Lyapunov}
\lim_{t\to\infty}\frac1t\int_{-\infty}^\infty \big(\log Z(t,x)\big)\phi(x)~dx=-\frac{1}{24}\int_{-\infty}^\infty\phi(x)~dx\qquad\text{in $L^2$},
\end{equation}
under the assumption that $Z(0,x)=e^{B(x)}$ where $B$ is a two-sided Brownian motion independent of $\xi$. 
Later, (1.4) is proved for more general initial data and $\phi = \delta(\cdot)$, see for example \cite{AmirCorwinQuastel}.
When combined with the moment Lyapunov exponents, the additional insight provided by the sample Lyapunov exponent
in \eqref{Equation: SHE Sample Lyapunov} is that the intermittent peaks only persist for a finite amount of time (since the
SHE's solution decays exponentially). This latter observation is in stark contrast with the intermittency
phenomenon observed in the parabolic Anderson model with time-independent noises; see, e.g., the monograph \cite{Konig}.

In recent years, the results in \cite{BertiniCancrini,BertiniGiacomin} were improved and extended in myriad ways.
For instance, \cite{DasTsai,GhosalLin} generalized \eqref{Equation: SHE Moment Lyapunov} to
all $n>0$ and an extensive class of initial conditions.
For the specific purposes of this paper, 
one important recent development
came from \cite{Chen}. More specifically, as shown in \cite[(3.2) and (4.1)]{Chen}, the Lyapunov exponents
of the SHE for integer powers admit the following variational interpretation:
\begin{equation}
\label{Equation: SHE Variational Lyapunov}
\mathfrak L_n=-\inf_{f\in\mathcal F_n}\left\{\frac12\int_{\mathbb R^n}|\nabla f(x)|^2~dx-\sum_{1\leq i<j\leq n}\int_{\mathbb R^n}\delta(x_i-x_j)f(x)^2~dx\right\},\qquad n\in \mathbb{Z}_{\geq 1},
\end{equation}
where $x=(x_1,\ldots,x_n)$, and $\mathcal F_n$ denotes the space of smooth rapidly-decreasing functions $f:\mathbb R^n\to\mathbb R$
such that $\|f\|_2=1$.
When combined with \eqref{Equation: SHE Moment Lyapunov}, the
equality in \eqref{Equation: SHE Variational Lyapunov} provided a rigorous proof of the fact that the ground state energy
of the Schr\"odinger operator on $\mathbb R^n$ defined as
\[H_n:=-\tfrac12\Delta-\sum_{1\leq i<j\leq n}\delta(x_i-x_j)\]
is equal to $-\mathfrak L_n=-\frac{ n(n^2-1)}{24}$. This confirmed physical predictions
made in \cite{Kardar,LiebLiniger}, which relied on the computation of the following ``ground state\footnote{Calling
$\Psi_n$ a ground state is an abuse of terminology, since although one can convincingly argue that $H_n\Psi_n=\frac{ n(n^2-1)}{24}\Psi_n$,
the norm $\|\Psi_n\|_2$ is infinite.}"
for $H_n$ using the Bethe ansatz:
\begin{equation}
\label{Equation: Ground State}
\Psi_n(x_1, \ldots, x_n) := \exp\left(\sum_{1 \leq i < j \leq n} -\frac{|x_i - x_j|}{2}\right).
\end{equation}

\subsection{Main Result}

In this note, we are interested in furthering the insights on the finer details of the geometry of
intermittent peaks hinted at by \eqref{Equation: SHE Variational Lyapunov} and \eqref{Equation: Ground State}.
More specifically, the proof of \eqref{Equation: SHE Variational Lyapunov} in \cite[Section 4]{Chen}
(most notably, its connection with Schr\"odinger semigroup theory \cite[Theorem 4.1.6]{ChenBook})
strongly suggests the following informal principle: The atypical configurations in $x\mapsto Z(t,x)$ that provide the main contributions to
$\mathbb E[Z(t,x)^n]$'s size as $t\to\infty$ should be closely related to the ground state $\Psi_n$.

Further evidence of this connection is provided by the fact that the multi-point correlation functions of the
SHE, which we define as
\[u_n(t,x)=\mathbb{E}\left[\prod_{i = 1}^n Z(t,x_i)\right],\qquad x=(x_1,\ldots,x_n)\in\mathbb R^n,\]
solve the delta-Bose gas PDEs
\[\partial_tu(t,x)=-H_nu(t,x),\qquad u(0,\cdot)=\delta(\cdot);\]
see \cite[Proposition 5.4.8]{borodin2014macdonald} and \cite[Corollary 1.7]{Nica},
and also \cite[Proposition 6.2.3]{borodin2014macdonald}.
An informal spectral expansion based on this fact suggests that if $t$ is large and $x_1,\ldots,x_n$
are all much smaller than $t$, then
\begin{align}
\label{Equation: Leading Eigenvalue Heuristic}
u_n(t,x)\approx e^{t\mathfrak L_n}\Psi_n(x).
\end{align}

\begin{remark}
\label{Remark: Strictly Sublinear is Optimal}
It is natural to expect that \eqref{Equation: Leading Eigenvalue Heuristic}
should only hold when the $x_i$'s are much smaller than $t$,
since otherwise we see from \eqref{Equation: Ground State}
that the product $e^{t\mathfrak L_n}\Psi_n(x)$ could
be of order $o(e^{t\mathfrak L_n})$. In such a case,
one expects that the leading order asymptotics of $u_n(t,x)$ are
not only explained by $H_n$'s ground state.
\end{remark}

In this context, our main result formalizes these heuristics as follows:

\begin{theorem}\label{thm:main}
Let $Z$ be as in \eqref{Equation: SHE} with $Z(0,\cdot)=\delta(\cdot)$.
Let $n\in \mathbb{Z}_{\geq 1}$, and for every $1\leq i\leq n$, let $\big(x_i(t)\big)_{t\geq0}$ be a strictly sublinear sequence of real numbers,
in the sense that $x_i(t)=o(t)$
as $t\to\infty$. It holds that
\begin{align}
\label{Equation: Main}
\mathbb{E}\left[\prod_{i = 1}^n Z\big(t,x_i(t)\big)\right]=\frac{(n-1)! \sqrt{2\pi}}{\sqrt{nt}}\exp\left(\mathfrak L_nt\right)\Psi_n\big(x_1(t), \ldots,x_n(t)\big)\big(1+o(1)\big)
\qquad\text{as }t\to\infty,
\end{align}
where $\mathfrak L_n$ is the Lyapunov exponent in \eqref{Equation: SHE Moment Lyapunov}
and
$\Psi_n$ is the ground state in \eqref{Equation: Ground State}.
\end{theorem}

Following-up on Remark \ref{Remark: Strictly Sublinear is Optimal}, one can guarantee
that the sublinearity assumptions on $x_i(t)$ in Theorem \ref{thm:main} are optimal thanks to \cite[Theorem 1.1]{Lin}.
Therein, exact asymptotics are derived for multi-point lyapunov exponents of the form
\[\lim_{t\to\infty}\frac1t\log\mathbb E\left[\prod_{i=1}^nZ(t,tx_i)\right],\qquad x_1,\ldots,x_n\in\mathbb R.\]
An examination of this result reveals that the scaling of $t$ in the space coordinates $x_i$ induce a rather different limiting behavior
that does not involve $\Psi_n$,
which is explained by an altogether different mechanism from the asymptotic
in Theorem \ref{thm:main} (i.e., \eqref{Equation: Leading Eigenvalue Heuristic}).

Finally, we note that the use of asymptotic spatial correlations to investigate
finer details of the geometry of intermittency (particularly the optimizers of the variational
problems that arise in the moment asymptotics) is not new to this paper. Notably, the work
\cite{GartnerHollander} proved a similar result for the discrete parabolic Anderson model
with time-independent noises.

\section{Outline of Proof}

In this section, we provide a bird's-eye view of the proof of Theorem \ref{thm:main};
the proofs of several technical propositions stated here are provided in later sections.
For this purpose, going forward, we assume that $n\in \mathbb{Z}_{\geq 1}$ and $x_1(t), \ldots, x_n(t)=o(t)$ are fixed.

The first and main technical ingredient in our proof consists of a contour integral representation
for the mixed moment $\mathbb{E}\big[\prod_{i = 1}^n Z\big(t,x_i(t)\big)\big]$, which allows to conveniently isolate the
leading order contribution of the latter.
Before we state this result, we need to introduce some notations:

\begin{defin}
 We say that a vector $\lambda$ is a partition of $n$, denoted by $\lambda \vdash n$, if $\lambda = (\lambda_1, \ldots, \lambda_\ell)$ with $\ell\geq1$,
 $\lambda_1, \ldots, \lambda_\ell \in \mathbb{Z}_{\geq 1}$, $\lambda_1 \geq \dots \geq \lambda_\ell$, and $\sum_{k = 1}^\ell \lambda_k = n$.
 
 Given $\lambda=(\lambda_1,\ldots,\lambda_\ell)\vdash n$, let
  $\ell(\lambda) := \ell$ denote the length of the partition,
 and define the combinatorial constant $m(\lambda) := m_1 ! m_2! \cdots m_n!$, where for every $1\leq i\leq n$, $m_i$ is the number of times that $i$ appears in $\lambda$.
 Finally, given a complex vector $\vec{w} =(w_1,\ldots,w_{\ell(\lambda)})\in\mathbb C^{\ell(\lambda)}$, we denote
\begin{equation}
\label{Equation: w circ lambda}
\vec w \circ \lambda := (w_1, w_1 + 1, \ldots, w_1 + \lambda_1 -1, w_2, w_2 + 1, \ldots, w_2 + \lambda_2 - 1, \ldots, w_{\ell(\lambda)}, \ldots,  w_{\ell(\lambda)} + \lambda_{\ell(\lambda)} - 1).
\end{equation}
Finally, let $S_n$ denote the set of permutations of $\{1,\ldots,n\}$.
\end{defin}

\begin{defin}
\label{Def: Et and nut}
For every $t\geq0$, let $x_{(1)}(t),\ldots,x_{(n)}(t)$ denote the coordinates $x_i(t)$ ordered in
nondecreasing order, i.e., $x_{(1)}(t)\leq x_{(2)}(t)\leq\cdots\leq x_{(n)}(t)$.
For every $t>0$, define the complex function $E_t:\mathbb C^n\to\mathbb C$ as
\begin{align}
\label{eq:E}
E_{t} (z_1, \ldots, z_n) 
&:= \sum_{\sigma \in S_n} \prod_{1 \leq 
	B < A \leq n} \frac{z_{\sigma(A)} - z_{\sigma (B)} - 1}{z_{\sigma(A)} - z_{\sigma(B)}} \exp\left(\sum_{i  =1}^n\left(\frac{t}{2} z_{\sigma(i)}^2 + x_{(i)}(t) z_{\sigma(i)}\right)\right).
\end{align}
Then, for every $\lambda\vdash n$, define
\begin{equation}\label{eq:nu}
\nu_\lambda(t):= \oint^{\otimes \ell(\lambda)}_{\gamma} \frac{1}{m(\lambda)} \det\Big[\frac{1}{w_i + \lambda_i  -w_j}\Big]_{i, j = 1}^{\ell(\lambda)}  E_t(\vec w\circ \lambda)\prod_{i = 1}^{\ell(\lb)} \frac{dw_i}{2\pi \mathbf{i}},
\end{equation}
where the contour $\gamma$ is given by $\theta + \mathbf{i} \mathbb{R}$ with arbitrarily fixed $\theta \in \mathbb{R}$,
and $\oint_\gamma^{\otimes\ell(\lambda)}$ denotes an $\ell(\lambda)$-fold contour integral on the same contour $\gamma$.
\end{defin}

Our main technical result is as follows:

\begin{prop}\label{prop:contexpan}
It holds that
\begin{equation}\label{eq:SHEcontour}
\mathbb{E}\left[\prod_{i = 1}^n Z\big(t,x_i(t)\big)\right] = 
\sum_{\lambda \vdash n} \nu_\lambda(t).
\end{equation}
\end{prop}

Proposition \ref{prop:contexpan} is proved in Section \ref{Section: Main Technical}.

\begin{remark}
The use of contour integral formulas to study moment asymptotics of the SHE (and other models) is not new to this paper.
In particular, \cite{borodin2014moments, corwin2020kpz} use the same idea
to investigate the one-point moment with a fixed space coordinate (i.e., the setting $x_1(t) = \cdots = x_n(t)=x$). In this scenario, the expression
\eqref{eq:E} in the contour integral formula is significantly simplified due to the symmetry with respect to the components $z_k$;
see, e.g., \cite[Proposition 6.2.7]{borodin2014macdonald}.
\end{remark}

With this in hand, Theorem \ref{thm:main}
relies on noting that the asymptotic contribution of the one-element partition $\lambda=(n)$
to the sum \eqref{eq:SHEcontour} gives the leading order term in \eqref{Equation: Main},
and that all the other summands grow at a slower rate:

\begin{prop}
\label{Proposition: Leading Order}
As $t\to\infty$, it holds that
\[ \nu_{(n)}(t)=\frac{(n-1)! \sqrt{2\pi}}{\sqrt{nt}}\exp\left(\mathfrak L_nt\right)\Psi_n\big(x_1(t), \ldots,x_n(t)\big)\big(1+o(1)\big).\]
\end{prop}

\begin{prop}
\label{Proposition: Remainder Terms}
For every $\lambda\vdash n$ such that $\lambda\neq(n)$, one has
\[v_\lambda(t)=o\big(v_{(n)}(t)\big)\qquad\text{as }t\to\infty.\]
\end{prop}

Theorem \ref{thm:main} then readily follows from \eqref{eq:SHEcontour}.

\section{Contour Integral Representation - Proof of Proposition \ref{prop:contexpan}}
\label{Section: Main Technical}

The starting point of the proof of Proposition \ref{prop:contexpan}
is the following lemma, which is stated in \cite[Proposition 6.2.3]{borodin2014macdonald}
and can be proved by combining \cite[Proposition 5.4.8]{borodin2014macdonald} and \cite[Corollary 1.7]{Nica}:

\begin{lemma}
\label{Lemma: Contour Formula 1}
It holds that
\begin{equation*}
\mathbb{E}\left[\prod_{i = 1}^n Z\big(t,x_i(t)\big)\right] =\oint_{\gamma_1} \cdots \oint_{\gamma_n} \left(\prod_{1 \leq i < j \leq n} \frac{z_i - z_j}{z_i - z_j - 1}\right)\left(\prod_{k = 1}^n e^{\tfrac{t}{2} z_k^2 +x_{(k)}(t) z_k}\right) \prod_{i = 1}^n \frac{dz_i}{2\pi \mathbf{i}},
\end{equation*}
where the contour of $z_k$ is given by $\gamma_k=a_k + \mathbf{i} \mathbb{R}$ for some real numbers $a_1, \ldots, a_n$ that can be arbitrary as long as they satisfy $a_j - a_{j+1} > 1$ for $j = 1, \ldots, n-1$. 
\end{lemma}

With this in hand, Proposition \ref{prop:contexpan} is an immediate consequence of the following claim,
which is inspired by a similar result that was claimed in \cite[Theorem 7.7]{Corwin2018Exactly}
(see also \cite[Proposition 3.2.1]{borodin2014macdonald} for a full proof
of \cite[Theorem 7.7]{Corwin2018Exactly}
in the special case of a symmetric function):


\begin{lemma}
	\label{Lemma: Contour Formula 2}
	Let $N\in \mathbb{Z}_{\geq 1}$ be arbitrary, and let
	$\gamma_1,\ldots,\gamma_N$ be contours of the form $\gamma_k=a_k+\mathbf i\mathbb R$
	for some $a_k\in\mathbb R$ such that $a_j - a_{j+1} > 1$.
	If $F:\mathbb C^N\to\mathbb C$ is analytic between the contours $\gamma_k$, then
	\begin{multline*}
	\oint_{\gamma_{1}} \cdots \oint_{\gamma_{N}}\left(\prod_{1 \leq i < j \leq N} \frac{z_i - z_j}{z_i - z_j - 1}\right) F(z_1,\ldots,z_N) \prod_{i = 1}^N \frac{dz_i}{2\pi \mathbf{i}}\\
	=\sum_{\lambda\vdash N}\oint^{\otimes \ell(\lambda)}_{\gamma} \frac{1}{m(\lambda)} \det\Big[\frac{1}{w_i + \lambda_i  -w_j}\Big]_{i, j = 1}^{\ell(\lambda)}E^F(\vec{w}\circ \lambda)\prod_{i = 1}^{\ell(\lb)} \frac{dw_i}{2\pi \mathbf{i}},
	\end{multline*}
	where the contour $\gamma$ is given by $\theta + \mathbf{i} \mathbb{R}$ with arbitrarily fixed $\theta \in \mathbb{R}$, and
	we denote
\begin{equation*}
E^{F} (z_1, \ldots, z_N) := \sum_{\sigma \in S_N}\left( \prod_{1 \leq 
		B < A \leq N} \frac{z_{\sigma(A)} - z_{\sigma (B)} - 1}{z_{\sigma(A)} - z_{\sigma(B)}}\right)F\big(z_{\sigma(1)},\ldots,z_{\sigma(N)}\big) .
\end{equation*}
\end{lemma}

\begin{proof}[Proof of Lemma \ref{Lemma: Contour Formula 2}]
	Let 
	\begin{equation}\label{eq:contourint}
	\mu_N^F := \oint_{\gamma_{1}} \cdots \oint_{\gamma_{N}}\left(\prod_{1 \leq i < j \leq N} \frac{z_i - z_j}{z_i - z_j - 1}\right) F(z_1,\ldots,z_N) \frac{dz_1}{2\pi \mathbf{i}}\cdots\frac{dz_N}{2\pi \mathbf{i}},
	\end{equation}
	and 
	\begin{equation}\label{eq:nuF}
	\nu^F_\lambda := \oint^{\otimes \ell(\lambda)}_{\gamma_{N}} \frac{1}{m(\lambda)} \det\Big[\frac{1}{w_i + \lambda_i  -w_j}\Big]_{i, j = 1}^{\ell(\lambda)}E^F(\vec{w}\circ \lambda)\prod_{i = 1}^{\ell(\lb)} \frac{dw_i}{2\pi \mathbf{i}},
	\end{equation}
	where we can take $\gamma=\gamma_N$.
	Therefore, it suffices to show that 
	\begin{equation*}
	\mu_N^F = \sum_{\lambda \vdash N} \nu_\lambda^F.
	\end{equation*} 
	We apply induction argument. It is clear that the desired equality holds when $N = 1$. Therefore, we only need to prove for every $n \in \mathbb{Z}_{\geq 1}$ that 
	\begin{equation}\label{eq:inducgoal}
	\mu^F_{n+1} = \sum_{\lambda \vdash n+1} \nu^F_{\lambda}
	\end{equation}
	knowing that the induction hypothesis holds, i.e.
	\begin{equation}\label{eq:ih} 
	\mu^F_{n} = \sum_{\lambda \vdash n} \nu^F_{\lambda}.
	\end{equation} 
	To prove \eqref{eq:ih}, we break up our proof into the following steps. 
	\bigskip
	\\
	\textbf{Step 1.}
	Let $G_{z_1} (z_2, \ldots, z_{n+1}) = F(z_1, \ldots, z_{n+1})$ and $g_{z_1} (z_2, \ldots, z_{n+1}) := \prod_{j = 2}^{n+1} \frac{z_1 - z_j}{z_1 - z_j - 1}$. By \eqref{eq:contourint}, it is straightforward to see that
	\begin{equation*} 
	\mu^F_{n+1} =  \oint_{\gamma_1} \frac{dz_1}{2\pi \mathbf{i}} \oint_{\gamma_2} \ldots \oint_{\gamma_{n+1}} \prod_{2 \leq i < j \leq n+1} \frac{z_i - z_j}{z_i - z_j - 1} (g_{z_1} G_{z_1}) (z_2, \ldots, z_{n+1}) \prod_{i = 2}^{n+1} \frac{dz_i}{2\pi \mathbf{i}}.
	\end{equation*}
	Applying \eqref{eq:ih} to the contour integral $ \oint_{\gamma_2} \ldots \oint_{\gamma_{n+1}}$, and then using the fact that $g_{z_1}$
	is symmetric to factor $E^{g_{z_1} G_{z_1}} = g_{z_1} E^{G_{z_1}}$, we obtain 
	\begin{equation*}
	\mu^F_{n+1} = \sum_{\lambda\vdash n}  \oint_{\gamma_1}  \bigg(\prod_{i  = 1}^{\ell(\lambda)} \frac{z_1 - w_i}{z_1 - w_i - \lambda_i}\bigg) \frac{dz_1}{2\pi \mathbf{i}} \oint_{\gamma_{n+1}}^{\otimes \ell(\lambda)} \frac{1}{m(\lambda)} \det\Big[\frac{1}{w_i + \lambda_i  -w_j}\Big]_{i, j = 1}^{\ell(\lambda)}  E^{G_{z_1}}(\vec{w} \circ \lambda)  \prod_{i = 1}^{\ell(\lambda)} \frac{dw_i}{2 \pi \mathbf{i}}.
	\end{equation*}
	Fix $\lambda$ and $w_1, \ldots, w_{\ell(\lambda)}$ and deform the contour of $z_1$ from $\gamma_1$ to $\gamma_{N+1}$. Doing so will cross simple poles at $\{w_i + \lambda_i\}_i$. By computing these residues, we get $U_{\lambda, 1}, \ldots, U_{\lambda, \ell(\lambda)}$ with 
	\begin{equation}\label{eq:U1}
	U_{\lambda, k} := \oint_{\gamma_{n+1}}^{\otimes \ell(\lambda)}  \bigg(\prod_{i = 1, i  \neq k }^{\ell(\lambda)} \frac{w_k + \lambda_k - w_i}{w_k + \lambda_k - w_i - \lambda_i}\bigg)  \frac{\lambda_k}{m(\lambda)} \det\Big[\frac{1}{w_i + \lambda_i  -w_j}\Big]_{i, j = 1}^{\ell(\lambda)} E^{G_{w_k + \lambda_k}} (\vec{w} \circ \lambda) \prod_{i = 1}^{\ell(\lambda)} \frac{dw_i}{2\pi \mathbf{i}}.
	\end{equation}
	When $\gamma_1$ arrives at $\gamma_{n+1}$ we let $z_1 = w_{\ell(\lambda) + 1}$ and obtain the contribution
	\begin{multline}\label{eq:U2}
	U_{\lambda, {\ell(\lambda) + 1}} :=   \oint_{\gamma_{n+1}}^{\otimes \ell(\lambda)+1} \bigg(\prod_{i  = 1}^{\ell(\lambda)} \frac{w_{\ell(\lambda) + 1} - w_i}{w_{\ell(\lambda) + 1} - w_i - \lambda_i}\bigg)\\\cdot\frac{1}{m(\lambda)} \det\Big[\frac{1}{w_i + \lambda_i  -w_j}\Big]_{i, j = 1}^{\ell(\lambda)}  E^{G_{w_{\ell(\lambda)+1}}} (\vec{w} \circ \lambda) 
	\prod_{i = 1}^{\ell(\lambda) + 1} 
	\frac{dw_i}{2\pi \mathbf{i}}.
	\end{multline}
	The above argument gives the identity 
	\begin{equation}\label{eq:step1end}
	\mu_{n+1}^F = \sum_{\lambda \vdash n} \sum_{k = 1}^{\ell(\lambda) + 1} U_{\lambda, k}.
	\end{equation}
	\textbf{Step 2.}
	We seek to rewrite the right hand side of \eqref{eq:step1end} into the right-hand side of \eqref{eq:inducgoal}. 
	Moreover, when $\lambda_i = \lambda_j$, we can obtain the integrand in $U_{\lambda, i}$ from that in $U_{\lambda, j}$ by swapping the variables $w_i$ and $w_j$. Since the contours for $w_i$ and $w_j$ are the same, we know that $U_{\lambda, i} = U_{\lambda, j}$.
	At this point, it is convenient for the purpose of this proof to use an alternate notation for partitions of $n$ and $n+1$,
	which is as follows:
	\begin{notation}
	We write any $\lambda\vdash n$ as $\lambda = a_1^{m_{a_1}} \ldots a_s^{m_{a_s}}$ where $a_1 > \ldots > a_s \geq 1$ and $m_{a_i}$ denotes the number of times that $a_i$ appears in $\lambda$. 
For any $\lb\vdash n+1$, we instead write $\lb = b_1^{m_{b_1}} \ldots b_r^{m_{b_r}}$ where $m_{b_i}$ has a similar meaning.  Furthermore, we let $M_{\lambda, k} = \sum_{i = 1}^k m_{a_i}$ and $M_{\lb, k} = \sum_{j = 1}^k m_{b_j}$.
	\end{notation}
Using this notation, we get that
	\begin{equation}\label{eq:Urelation}
	\sum_{\lambda \vdash n} \sum_{k = 1}^{\ell(\lambda)+1} U_{\lambda, k} = \sum_{\lambda \vdash n} \bigg(\Big(\sum_{i = 1}^s m_{a_i} U_{\lambda, M_{\lambda, i-1}+1}\Big) + U_{\lambda, \ell(\lambda)+1}\bigg).
	\end{equation}
	 Let us define a bijection $f$ from $\{(\lambda, i): \lambda \vdash n, i \in \{1, \ldots, s+1\}\}$ to $\{(\lb, j): \lb \vdash n+1, j \in \{1, \ldots, r\}\}$. Given $(\lambda, i)$, we define $(\lb, j) = f(\lambda, i)$ by the following rule: $\lb$ is defined via adding the $M_{\lambda, i-1}+1$-th component of $\lambda$ by $1$. $j$ is defined to be the unique number satisfying $M_{\lb, j} = M_{\lambda, i-1}+1$. It is clear that $f$ is invertible and $(\lambda, i) = f^{-1}(\lb, j)$ is given by the following rule: $\lambda$ is obtained from $\lb$ via subtracting the $M_{\lb, j}$-th component by $1$ and $i$ is defined to be the unique number satisfying $M_{\lambda, i-1} + 1 = M_{\lb, j}$. 
	
	For $j \in \{1, \ldots, r\}$, we let $\lb[k]$ be the vector obtained via subtracting the $k$-th component of $\lb$ by $1$. Note that $\lb[k]$ is a partition if and only if $k = M_{\lb, j}$ for some $j$. From the previous paragraph, it is not hard to verify that if $(\lb, j) = f(\lambda, i)$, then
	\begin{equation}\label{eq:mrelation}
	m_{a_i} = m_{b_{j+1}} \mathbf{1}_{\{b_j - 1 = b_{j+1}\}} + 1  \text{ when } i \in \{1, \ldots, s\}, \qquad  1 = m_{b_{j+1}} \mathbf{1}_{\{b_j - 1 = b_{j+1}\}}+1 \text{ when } i = s+1.
	\end{equation}
Moreover, one can verify that $M_{\lb, j} = M_{\lambda, i-1} + 1$ and $\lb[M_{\lb, j}] = \lambda$, which yields 
	\begin{equation*}
	U_{\lambda, M_{\lambda, i-1} + 1} = U_{\lb [M_{\lb, j}], M_{\lb, j}}.
	\end{equation*}
	The above equality and \eqref{eq:mrelation} imply that 
	\begin{equation*}
	\sum_{\lambda \vdash n} \bigg(\Big(\sum_{i = 1}^s m_{a_i} U_{\lambda, M_{\lambda, i-1} + 1}\Big) + U_{\lambda, \ell(\lambda)+1}\bigg) = \sum_{\lb \vdash n+1} \sum_{j = 1}^r (m_{b_{j+1}} \mathbf{1}_{\{b_j - 1 = b_{j+1}\}} + 1) U_{\lb [M_{\lb, j}], M_{\lb, j}}.
	\end{equation*}
	By the above equality and \eqref{eq:Urelation} and \eqref{eq:step1end}, we obtain 
	\begin{equation}\label{eq:step2end}
	\mu_{n+1}^F = \sum_{\lb \vdash n+1} \sum_{j = 1}^r (m_{b_{j+1}} \mathbf{1}_{\{b_j - 1 = b_{j+1}\}} + 1) U_{\lb [M_{\lb, j}], M_{\lb, j}}.
	\end{equation}
	\textbf{Step 3. } By \eqref{eq:step2end}, the proof of \eqref{eq:inducgoal} reduces to showing that for all $\lb \vdash n+1$, 
	\begin{equation}\label{eq:step3goal}
	\nu^F_{\lb} = \sum_{k = 1}^r (m_{b_{k+1}} \mathbf{1}_{\{b_k - 1 = b_{k+1}\}} + 1) U_{\lb [M_{\lb, k}], M_{\lb, k}}.
	\end{equation}
	The rest of the proof is devoted to proving \eqref{eq:step3goal}. For $j \in \{1, \ldots, n+1\}$, we define $\Lambda (k) := \sum_{i = 1}^k \lb_i$ and
	\begin{equation*}
	E^F_j (z_1, \ldots, z_{n+1}) := \sum_{\sigma \in S_{n+1}, \sigma(1) = j} \left( \prod_{1 \leq 
		B < A \leq n+1} \frac{z_{\sigma(A)} - z_{\sigma (B)} - 1}{z_{\sigma(A)} - z_{\sigma(B)}}\right)F\big(z_{\sigma(1)},\ldots,z_{\sigma(n+1)}\big).
	\end{equation*}
	Note that $E^F = \sum_{j = 1}^{n+1} E_j^F$ and $E^F_j (\vec{w} \circ \lb)$ is non-zero only if $j \in \{\Lambda (1), \ldots, \Lambda (\ell(\lb))\}$. By \eqref{eq:nuF}, we know that 
	\begin{equation*}
	\nu^F_{\lb} := \sum_{k = 1}^{\ell(\lb)}\oint^{\otimes \ell(\lambda)}_{\gamma_{n+1}} \frac{1}{m(\lb)} \det\Big[\frac{1}{w_i + \lb_i  -w_j}\Big]_{i, j = 1}^{\ell(\lb)}E_{\Lambda(k)}^F(\vec{w}\circ \lb)\prod_{i = 1}^{\ell(\lb)} \frac{dw_i}{2\pi \mathbf{i}}.
	\end{equation*}
	For $k \in \{1, \dots, \ell(\lb)\}$, define
	\begin{equation*}
	f_{\lb, k} (\vec{w}) := \prod_{i = 1, i\neq k}^{\ell(\lb)} \frac{w_k + \lb_k - w_i}{w_k + \lb_k - w_i - \lb_i}. 
	\end{equation*} 
It is straightforward to verify that 
	\begin{align*}
	E_{\Lambda(k)}^F(\vec{w}\circ \lb)=  \lb_k f_{\lb, k} (\vec{w}) E^{G_{w_k + \lb_k - 1}} (\vec{w} \circ \lb[k]). 
	\end{align*} 
	Therefore, we have 
	\begin{equation*}
	\nu_{\lb}^F = \sum_{k = 1}^{\ell(\lb)} \oint_{\gamma_{n+1}}^{\otimes \ell(\lb)} \frac{\lb_k}{m(\lb)} f_{\lb, k} (\vec{w}) \det\Big[\frac{1}{w_i + \lb_i - w_j}\Big]_{i, j = 1}^{\ell(\lb)} E^{G_{w_k + \lb_k - 1}} (\vec{w} \circ \lb[k]) \prod_{i = 1}^{\ell(\lb)} \frac{dw_i}{2\pi \mathbf{i}}. 
	\end{equation*}
	One can see that if $\lb_i = \lb_j$ for some $i, j \in \{1, \ldots, \ell(\lb)\}$, then $i$-th integrand on the right hand side above can be obtained by swapping the variables $w_i$ and $w_j$ in the $j$-th integrand. As a consequence, the $i$-th integral in the summation above is equal to the $j$-th integral. 
	Merging the term with the same values in the summation, we have 
	\begin{equation}\label{eq:nuexpansion}
	\nu_{\lb}^F = \sum_{k = 1}^r \oint_{\gamma_{n+1}}^{\otimes \ell(\lb)} \frac{m_{b_k}}{m(\lb)} \lb_{M_{\lb, k}} f_{\lb, M_{\lb, k}}(\vec{w}) \det\Big[\frac{1}{w_i + \lb_i - w_j}\Big]_{i, j = 1}^{\ell(\lb)} E^{G_{w_{M_{\lb, k} + \lb_{M_{\lb, k}} - 1}}} (\vec{w} \circ \lb[M_{\lb, k}]) \prod_{i = 1}^{\ell(\lb)} \frac{dw_i}{2\pi \mathbf{i}}.
	\end{equation} 
	It is standard to check that 
	\begin{equation*}
	\frac{m_{b_k}}{m(\lb)} = \frac{m_{b_{k+1}} \mathbf{1}_{\{b_k - 1 = b_{k+1}\}} + 1}{m(\lb[M_{\lb, k}])} \text{ and } \lb_{M_{\lb, k}} - 1 = \lb[M_{\lb, k}]_{M_{\lb, k}}.
	\end{equation*}
	This, together with \eqref{eq:nuexpansion} imply that
	\begin{align}\label{eq:step3nu}
	\nu^F_{\lb} = \sum_{k = 1}^r (m_{b_{k+1}} \mathbf{1}_{\{b_k - 1 = b_{k+1}\}} + 1) &\oint_{\gamma_{n+1}}^{\ell(\lb)} \frac{1}{m(\lb[M_{\lb, k}])} \lb_{M_{\lb, k}} f_{\lb, M_{\lb, k}} (\vec{w}) \det\Big[\frac{1}{w_i + \lb_i - w_j}\Big]_{i, j = 1}^{\ell(\lb)}\\ 
	\notag
	&\hspace{1em} \times E^{G_{w_k + \lb[M_{\lb, k}]_{M_{\lb, k}}}} (\vec{w} \circ \lb[M_{\lb, k}]) \prod_{i = 1}^{\ell(\lb)} \frac{dw_i}{2\pi \mathbf{i}}.
	\end{align} 
	To prove \eqref{eq:step3goal}, it suffices to show that for $k \in \{1, \ldots, r\}$,
	\begin{equation*}
	\text{The $k$-th integral in \eqref{eq:step3nu}} = U_{\lb [M_{\lb, k}], M_{\lb, k}}.
	\end{equation*}
	It suffices to match the integrands in the integral on both sides. If $\lb_{M_{\lb, k}} > 1$, then $U_{\lb [M_{\lb, k}], M_{\lb, k}}$ takes the form of \eqref{eq:U1}. 
	By Cauchy determinant formula, one can verify that 
	\begin{align}\notag
	&\lb_{M_{\lb, k}} f_{\lb, M_{\lb, k}} (\vec{w}) \det\Big[\frac{1}{w_i + \lb_i - w_j}\Big]_{i, j = 1}^{\ell(\lb)}\\ 
\notag
	&=  \lb[M_{\lb, k}]_{M_{\lb, k}} \det\Big[\frac{1}{w_i + \lb[M_{\lb, k}]_i  -w_j}\Big]_{i, j = 1}^{\ell(\lb[M_{\lb, k}])} \prod_{i = 1, i \neq M_{\lb, k}}^{\ell(\lb[M_{\lb, k}])} \frac{w_{M_{\lb, k}} + \lb[M_{\lb, k}]_{M_{\lb, k}} - w_i}{w_{M_{\lb, k}} + \lb[M_{\lb, k}]_{M_{\lb, k}} - w_i - \lb[M_{\lb, k}]_i}.
	\end{align} 
	Applying this equality to the integrand in \eqref{eq:step3nu}, we see that the integrand in \eqref{eq:step3nu} is equal to that in $U_{\lb[M_{\lb, k}], M_{\lb, k}}$. If instead $\lb_{M_{\lb, k}} = 1$, then $U_{\lb [M_{\lb, k}], M_{\lb, k}}$ takes the form of \eqref{eq:U2}, a similar argument concludes the matching of the integrands. 
\end{proof}


%

\section{Leading Order Term - Proof of Proposition \ref{Proposition: Leading Order}}

By definition of $\nu_\lambda(t)$ in \eqref{eq:nu}, if we take $\lambda=(n)$
and make the choice $\theta=0$, then we have that
\begin{align}
\label{Equation: Leading Order Term 1}
\nu_{(n)}(t)=\oint_{\mathbf i\mathbb R} \frac{1}{1!}\cdot\frac{1}{n}\cdot E_t(w,w+1,\ldots,w+n-1)\,\frac{dw}{2 \pi \mathbf{i}}.
\end{align}
Recall the definition of $E_t$ in \eqref{eq:E}.
If a permutation $\sigma\in S_n$ is such that
there exists indices $1\leq \beta<\alpha\leq n$ with $\sigma(\alpha)=\sigma(\beta)+1$,
then the product
\[\prod_{1 \leq B < A \leq n} \frac{z_{\sigma(A)} - z_{\sigma (B)} - 1}{z_{\sigma(A)} - z_{\sigma(B)}}\]
contains the factor
\[z_{\sigma(\alpha)} - z_{\sigma (\beta)} - 1=z_{\sigma(\beta)+1} - (z_{\sigma (\beta)} + 1).\]
Any such term vanishes if we take
\[(z_1,\ldots,z_n)=(w,w+1,\ldots,w+n-1),\]
as $z_{k+1}=z_k+1$ for all $1\leq k\leq n-1$, including $k=\sigma(\beta)$.
Thus, the only permutations that can contribute to
\[E_t(w,w+1,\ldots,w+n-1)\]
in the sum \eqref{eq:E} are those such that $\sigma(\alpha)\neq\sigma(\beta)+1$ whenever $\beta<\alpha$.
The only partition that satisfies this property is $\sigma(i)=n+1-i$, $1\leq i\leq n$.
With this particular choice of permutation, we note that
\[\frac1n\prod_{1 \leq B < A \leq n} \frac{z_{\sigma(A)} - z_{\sigma(B)} - 1}{z_{\sigma(A)} - z_{\sigma(B)}}
=\frac1n\prod_{1 \leq B < A \leq n}\frac{z_{n+1-A} - (z_{n+1-B} + 1)}{z_{n+1-A} - z_{n+1-B}},\]
and
\[\exp\left(\sum_{i  =1}^n\left(\frac{t}{2} z_{\sigma(i)}^2 + x_{(i)}(t) z_{\sigma(i)}\right)\right)
=\exp\left(\sum_{i  =1}^n\left(\frac{t}{2} z_{n+1-i}^2 + x_{(i)}(t) z_{n+1-i}\right)\right).\]
With the choice of coordinates $z_i=w+i-1$, this now becomes
\[\frac1n\prod_{1 \leq B < A \leq n} \frac{z_{\sigma(A)} - z_{\sigma(B)} - 1}{z_{\sigma(A)} - z_{\sigma(B)}}
=\frac1n\prod_{1 \leq B < A \leq n}\frac{1+(A-B)}{A-B}=(n-1)!\]
and by expanding and completing the square,
\begin{multline*}
\exp\left(\sum_{i  =1}^n\left(\frac{t}{2} z_{\sigma(i)}^2 + x_{(i)}(t) z_{\sigma(i)}\right)\right)
=\exp\left(\sum_{i  =1}^n\left(\frac{t}{2} (w+n-i)^2 + x_{(i)}(t)\big(w+n-i\big)\right) \right)\\
=\exp\left(\left(\frac{n(n^2-1)t}{24}+\sum_{i=1}^nx_{(i)}(t)\left(\frac{n+1}{2}-i\right)-\frac{(\sum_{i=1}^nx_{(i)}(t))^2}{2nt}+\frac{nt}{2} \left(w-\left(\frac{1-n}{2}-\frac{\sum_{i=1}^nx_{(i)}(t)}{nt}\right)\right)^2\right)\right).
\end{multline*}
At this point, we note that $\exp(\frac{n(n^2-1)t}{24})=\exp(\mathfrak L_n t)$,
\[\exp\left(\sum_{i=1}^nx_{(i)}(t)\left(\frac{n+1}{2}-i\right)\right)
=\exp\left(\sum_{1\leq i<j\leq n}\frac{x_{(j)}(t)-x_{(i)}(t)}2\right)
=\Psi\big(x_1(t),\ldots,x_n(t)\big),\]
and $-\frac{1}{2nt}(\sum_{i=1}^nx_{(i)}(t))^2=o(t)$ because $x_i(t)=o(t)$.
Thus, by computing a straightforward Gaussian integral (with variance $\frac1{nt}$),
the integral in \eqref{Equation: Leading Order Term 1} yields
\begin{align*}
\nu_{(n)}(t)=\frac{(n-1)! \sqrt{2\pi}}{\sqrt{nt}}\exp\left(\mathfrak L_nt\right)\Psi_n\big(x_1(t), \ldots,x_n(t)\big)\big(1+o(1)\big)\qquad
\text{as }t\to\infty
\end{align*}
as desired.

\section{Remainder Terms - Proof of Proposition \ref{Proposition: Remainder Terms}}

In the argument that follows, we use $C,c>0$ to denote positive constants independent of $t$
(but which may depend on some other parameters, such as $n$)
whose values may change from one display to the next.
Moreover, we use $\Re$ and $\Im$ to respectively denote the real and imaginary parts of
a complex number.

Let $\lambda\vdash n$ be a fixed partition such that $\lambda\neq(n)$.
Our aim is to control the moduli of the functions whose products appear inside the integral \eqref{eq:nu} individually
(i.e., the determinant and the multiple products in $E_t(\vec w\circ\lambda)$), and thus obtain a result that grows
at a slower rate than $\nu_{(n)}(t)$ as $t\to\infty$. In this context, looking back at \eqref{eq:E},
we see that the terms that are the most difficult to control are those that appear in $E_t(\vec w\circ\lambda)$
due to the divisions by $z_{\sigma(A)}-z_{\sigma(B)}$. In order to get around this, we deform the $\ell(\lambda)$ contours
in \eqref{eq:nu} as follows:
\begin{equation}\label{eq:nu deformed}
\nu_\lambda(t)= \oint_{\gamma^\varepsilon_1}\cdots\oint_{\gamma^\varepsilon_{\ell(\lambda)}} \frac{1}{m(\lambda)} \det\Big[\frac{1}{w_i + \lambda_i  -w_j}\Big]_{i, j = 1}^{\ell(\lambda)}  E_t(\vec w\circ \lambda)\prod_{i = 1}^{\ell(\lb)} \frac{dw_i}{2\pi \mathbf{i}},
\end{equation}
where $\varepsilon\in(0,\frac1{n-1})$ is a fixed constant, $\gamma^\varepsilon_1=\gamma=\theta+\mathbf i\mathbb R$,
and for every $1<k\leq\ell(\lambda)$, we define $\gamma^\varepsilon_k=\theta+(k-1)\varepsilon+\mathbf i\mathbb R$.
We recall that $\theta\in\mathbb R$ can be chosen arbitrarily (as per Definition \ref{Def: Et and nut}). Moreover,
the restriction $\varepsilon<\frac1{n-1}$ ensures that no poles in the determinant are crossed when
deforming \eqref{eq:nu} into \eqref{eq:nu deformed} because $\theta+(k-1)\varepsilon<\theta+1\leq\theta+\lambda_i$ for any choice of
$\lambda$ and $1\leq k,i\leq\ell(\lambda)$; otherwise $\varepsilon$ can be chosen arbitrarily at this time.

With this in hand, we first note that since $|z^{-1}|\leq|\Re(z)|^{-1}$,
for every $\theta\in\mathbb R$ and $\varepsilon\in(0,\frac1{n-1})$, there exists some $C>0$ such that
\[\sup_{w_1\in\gamma^\varepsilon_1,\ldots,w_{\ell(\lambda)}\in\gamma^\varepsilon_{\ell(\lambda)}}\left|\frac{1}{m(\lambda)} \det\Big[\frac{1}{w_i + \lambda_i  -w_j}\Big]_{i, j = 1}^{\ell(\lambda)}\right|\leq C\]
and
\[\sup_{w_1\in\gamma^\varepsilon_1,\ldots,w_{\ell(\lambda)}\in\gamma^\varepsilon_{\ell(\lambda)}}\left|\prod_{1 \leq B < A \leq n}\frac{z_{\sigma(A)} - z_{\sigma (B)} - 1}{z_{\sigma(A)} - z_{\sigma(B)}}\right|\leq C\qquad\text{for all }\sigma\in S_n.\]
Secondly, given that $|e^{z}|=e^{\Re(z)}\leq e^{|\Re(z)|}$ and $\Re(z^2)=\Re(z)^2-\Im(z)^2$, for any $z\in\mathbb C$ and $1\leq i\leq n$,
\[\left|\exp\left(\tfrac{t}{2} z^2 + x_{(i)}(t) z\right)\right|
=\exp\left(\tfrac t2\Re(z^2)+x_{(i)}(t)\Re(z)\right)
\leq\exp\left(\tfrac t2\Re(z)^2-\tfrac t2\Im(z)^2+\mathsf m(t)|\Re(z)|\right),\]
where we denote
\begin{align}
\notag
\mathsf m(t)=\max\{|x_i(t)|:1\leq i\leq n\}=o(t).
\end{align}
Therefore, since the real parts of the components of $\vec w\circ\lambda$ in \eqref{Equation: w circ lambda}
are bounded for any fixed $\theta$ and $\varepsilon$,
there exists some $c>0$ such that $|\Re(z)|\leq c$ for all $z\in \vec w\circ\lambda$; hence
it suffices to prove that there exists a choice of $\theta\in\mathbb R$ and $\varepsilon\in(0,\frac1{n-1})$ such that
\begin{equation*}
e^{c\mathsf m(t)}\oint_{\gamma^\varepsilon_1}\cdots\oint_{\gamma^\varepsilon_{\ell(\lambda)}}\exp\left(\sum_{k = 1}^{\ell(\lambda)}\sum_{i = 1}^{\lambda_k} \frac t2\Re(w_k+i-1)^2-\frac t2\Im(w_k+i-1)^2\big) \right)\prod_{i = 1}^{\ell(\lb)} \frac{dw_i}{2\pi \mathbf{i}}=o\big(v_{(n)}(t)\big)
\end{equation*}
as $t\to\infty$.
Given that
\[\Re(w_k+i-1)^2=\big(\theta+(k-1)\varepsilon+i-1\big)^2\qquad\text{and}\qquad\Im(w_k+i-1)^2=\Im(w_k)^2,\]
by a Gaussian integral
\begin{multline}\label{eq:nu deformed 2}
e^{c\mathsf m(t)}\oint_{\gamma^\varepsilon_1}\cdots\oint_{\gamma^\varepsilon_{\ell(\lambda)}}\exp\left(\sum_{k = 1}^{\ell(\lambda)}\sum_{i = 1}^{\lambda_k} \frac t2\Re(w_k+i-1)^2-\frac t2\Im(w_k+i-1)^2\big) \right)\prod_{i = 1}^{\ell(\lb)} \frac{dw_i}{2\pi \mathbf{i}}\\
\leq \frac{C}{\sqrt t}\exp\left(c\mathsf m(t)+t\sum_{k = 1}^{\ell(\lambda)}\sum_{i = 1}^{\lambda_k} \frac{\big(\theta+(k-1)\varepsilon+i-1\big)^2}{2}\right).
\end{multline}
Thus, we need only prove that the right-hand side of \eqref{eq:nu deformed 2} is of order $o\big(v_{(n)}(t)\big)$
as $t\to\infty$ for any $C,c>0$.

Toward this end, we note that since $e^{c\mathsf m(t)},\Psi_n\big(x_1(t), \ldots,x_n(t)\big)=e^{o(t)}$, it suffices to find a combination
of $\theta\in\mathbb R$ and $\varepsilon\in(0,\frac1{n-1})$ such that
\[\sum_{k = 1}^{\ell(\lambda)}\sum_{i = 1}^{\lambda_k} \frac{\big(\theta+(k-1)\varepsilon+i-1\big)^2}{2}<\mathfrak L_n.\]
By expanding the square above, we get that
\[\sum_{k = 1}^{\ell(\lambda)}\sum_{i = 1}^{\lambda_k} \frac{\big(\theta+(k-1)\varepsilon+i-1\big)^2}{2}\leq \sum_{k = 1}^{\ell(\lambda)}\sum_{i = 1}^{\lambda_k} \frac{\big(\theta+i-1\big)^2}{2}+C\varepsilon\]
for some constant $C>0$ that depends on $\lambda$ and $\theta$, but is independent of $\varepsilon$.
Given that $\varepsilon$ can be taken arbitrarily small, it suffices to prove that there exists some $\theta\in\mathbb R$ such that
\[\sum_{k = 1}^{\ell(\lambda)}\sum_{i = 1}^{\lambda_k} \frac{\big(\theta+i-1\big)^2}{2}<\mathfrak L_n.\]
By Faulhaber's formula and the fact that $\sum_{k=1}^{\ell(\lambda)}\lambda_k=n$,
\begin{multline*}
\sum_{k = 1}^{\ell(\lambda)}\sum_{i = 1}^{\lambda_k} \frac{\big(\theta+i-1\big)^2}{2}
=\sum_{k = 1}^{\ell(\lambda)}\frac{\left(6\lambda_k\theta^2+6(\lambda_k^2-\lambda_k)\theta+2\lambda_k^3-3\lambda_k^2+\lambda_k\right)}{12}\\
=\frac{(6n\theta^2+6(\sum_{k=1}^{\ell(\lambda)}\lambda_k^2-n)\theta+2\sum_{k=1}^{\ell(\lambda)}\lambda_k^3-3\sum_{k=1}^{\ell(\lambda)}\lambda_k^2+n)}{12}.
\end{multline*}
By elementary calculus, it is easy to see that the above is minimized (with respect to $\theta$) at
\[\theta_\star=\frac{n-\sum_{k=1}^{\ell(\lambda)}\lambda_k^2}{2 n}.\]
With this particular choice, we the sum simplifies to
\[\sum_{k = 1}^{\ell(\lambda)}\sum_{i = 1}^{\lambda_k} \frac{\big(\theta_\star+i-1\big)^2}{2}=\frac{1}{24} \left(4\sum_{k=1}^{\ell(\lambda)}\lambda_k^3-\frac3n \left(\sum_{k=1}^{\ell(\lambda)}\lambda_k^2\right)^2-n\right).\]
Note that the above reduces to $\mathfrak L_n$ when $\lambda=(n)$; it now only remains to show that
any other choice of permutation yields a quantity that is strictly smaller, that is,
\[\max_{\lambda\vdash n,~\lambda\neq(n)}\left\{4\sum_{k=1}^{\ell(\lambda)}\lambda_k^3-\frac3n \left(\sum_{k=1}^{\ell(\lambda)}\lambda_k^2\right)^2\right\}<n^3\qquad\text{for every }n\geq1.\]
For this, we recall the elementary inequality $\|x\|_{\ell^p}\leq\|x\|_{\ell^2}$ for all $x\in\mathbb R^N$ and $p>2$, which is achieved only
if all components of $x$ except one are equal to zero. Since $\lambda\neq(n)$, the components $\lambda_k$ cannot all be equal, whence
\[\max_{\lambda\vdash n,~\lambda\neq(n)}\left\{4\sum_{k=1}^{\ell(\lambda)}\lambda_k^3-\frac3n \left(\sum_{k=1}^{\ell(\lambda)}\lambda_k^2\right)^2\right\}<
\max_{\lambda\vdash n,~\lambda\neq(n)}\left\{4\left(\sum_{k=1}^{\ell(\lambda)}\lambda_k^2\right)^{3/2}-\frac3n \left(\sum_{k=1}^{\ell(\lambda)}\lambda_k^2\right)^2\right\}
\leq\max_{r\in\mathbb R}\left\{4r^3-\frac3nr^4\right\}.\]
It is easily seen (using elementary calculus) that the maximum of the function on the right-hand side
is $n^3$, which is achieved at $r=n$. Thus, the proof of Proposition \ref{Proposition: Remainder Terms} (and therefore of Theorem \ref{thm:main}) is complete.

\bibliographystyle{alpha}
\bibliography{ref}
\end{document}